\documentclass[11pt,makeidx, reqno]{amsart}

\usepackage{amsthm,amsfonts,amssymb,amsmath,amscd,amsrefs,thmtools}
\usepackage{mathtools}
\usepackage[all]{xy}
\usepackage{xr-hyper}
\usepackage{color}
\usepackage{bbm}
\usepackage{amssymb}
\usepackage{xfrac}
\usepackage{thm-restate}
\definecolor{crimsonglory}{rgb}{0.75, 0.0, 0.2}
\definecolor{darkblue}{rgb}{0.0, 0.0, 0.55}
	\definecolor{deepskyblue}{rgb}{0.0, 0.75, 1.0}

\usepackage[colorlinks=true,linkcolor = blue, citecolor = deepskyblue]{hyperref}

\usepackage[OT2,OT1]{fontenc}
\usepackage{mathrsfs}
\usepackage{multirow}
\usepackage{graphicx}
\usepackage{enumitem}

\newtheoremstyle{mystyle}
  {8pt}% measure of space to leave above the theorem. E.g.: 3pt
  {3pt}% measure of space to leave below the theorem. E.g.: 3pt
  {\em}% name of font to use in the body of the theorem
  {}% measure of space to indent
  {\scshape}% name of head font
  {.}% punctuation between head and body
  {3pt}% space after theorem head; " " = normal interword space
  {}% Manually specify head

\theoremstyle{mystyle}

\usepackage{xcolor}
\definecolor{bg}{RGB}{45,34,25}      % pick one row’s background
\definecolor{txt}{RGB}{230,210,180} 
\newtheorem{theorem}{Theorem}[section]

\newtheorem{thm}[theorem]{\textsc{Theorem}}

\newtheorem{lemma}[theorem]{\textsc{Lemma}}
\newtheorem{prop}[theorem]{\textsc{Proposition}}

\theoremstyle{definition}

\usepackage{color}

\title[]{The least quadratic residue and \\ integers represented by quadratic forms}
\author{K. Soundararajan}
\address{Department of Mathematics, Stanford University, Stanford, CA 94305, USA}
\author{Jo\~ao C. C. Vargas}
\address{Department of Mathematics, University of Toronto, Toronto, ON M5S 2E4, Canada}

\dedicatory{To Roger Heath-Brown on the occasion of  his 75th birthday}
\thanks{The first author is partially supported through NSF grant DMS-2100933}

\begin{document}

\begin{abstract} Let $\ell(n)$ denote the least non-trivial reduced quadratic residue modulo $n$; that is, 
$\ell(n)$ denotes the smallest square-free integer $r>1$ with $(r,n)=1$ and $r\equiv x^2 \bmod {n}$.   
We establish nearly optimal bounds for $\ell(n)$, both in terms of the magnitude of $n$ and of its number of prime factors $\omega(n)$. In particular, we construct moduli $n$ for which $\ell(n)$ is unexpectedly large. As an application of our results, we prove bounds for the rate at which binary quadratic forms with bounded discriminant represent all positive integers up to $N$.
 \end{abstract}
 
\maketitle

\section{Introduction}\label{intro}

\noindent Given a positive integer $n$, we say that $r$ is a quadratic residue $\bmod\  {n}$ if $r\equiv x^2 \bmod {n}$ for some integer $x$.  This paper studies $\ell(n)$ which denotes the least square-free integer $r>1$ with $(r,n)=1$ and $r$ being a quadratic residue $\bmod\ n$.  Equivalently,    
\begin{equation}\label{l(n)Defn}
    \ell(n) := \min\{ r \in \mathbb{N} \setminus \{1, 4, 9, \dots\}: r \equiv x^2 \bmod{n},\, \gcd(r, n) = 1\}.
\end{equation}
 For instance, if $p >5$ is a prime number, then note that $\ell(p)\le 6$ since one of $2$, $3$ or $6$ must be a quadratic residue $\bmod \ p$.  One of our goals will be to obtain bounds for $\ell(n)$, both in terms of the size of $n$ and also in terms of the number of prime factors of $n$.  We shall also discuss applications of a closely related problem to the study of representing integers by a family of binary quadratic forms.

 Write the prime factorization of $n$ as $n=2^e p_1^{e_1} \cdots p_k^{e_k}$ where the $p_j$ are odd primes with exponents $e_j\ge 1$, and the exponent of $2$ is $e\ge 0$.   Then a reduced residue $r \bmod n$ is a quadratic residue precisely when $(\frac{r}{p_j}) =1$ for all $1\le j\le k$, and $r\equiv 1 \bmod{2^e}$ when $0\le e\le 3$, and $r\equiv 1 \bmod{8}$ for $e\ge 3$. Thus in studying this problem one may assume that all odd primes dividing $n$ occur to exponent $1$, and that $2$ occurs to exponent at most $3$.   Further, the quadratic residues $r$ lie in a set of density $\asymp 2^{-k}$.

 \begin{thm}\label{thm1.1}  If $k$ denotes the number of odd prime factors of $n$ then 
    \[
    \ell(n) \le Ck^2 4^k,
    \]
    for an absolute constant $C$.  Further, for each $k\ge 2$ there exist infinitely many odd integers $n$ with exactly $k$ distinct prime factors satisfying 
    \[
    \ell(n) \ge 4^{k- Ck/\log k},
    \]
    for some absolute constant $C$.  
\end{thm}

The upper bound above is a straightforward application of the pigeonhole principle.  The surprising part is the $\Omega$-result: since the density of the set of quadratic residues is $\asymp 2^{-k}$, one might have expected $\ell(n)$ to be not much larger than $2^k$.   The construction of $n$ with large $\ell(n)$ is motivated by ideas from error correcting codes, see \cite{boseClass, hocquenghenCodes}.

Since $\omega(n)$ (the number of distinct prime factors of $n$) is bounded by $\log n/\log \log n + O(\log n/(\log \log n)^2)$ the upper bound in Theorem \ref{thm1.1} implies that 
\begin{equation} \label{1.2}
\ell(n) \le \exp\Big( \Big( \log 4 +O\Big(\frac{1}{\log \log n}\Big) \Big)\frac{\log n}{\log \log n}\Big). 
\end{equation}
Our next result furnishes large values of $\ell(n)$ in terms of $n$; we state it in a form which is more precise and allows applications to binary quadratic forms.

\begin{thm}\label{thm1.2}  Let $\Delta \ge 10$ be given.   There exists an odd square-free integer $n\le N=N(\Delta)$ such that for every fundamental discriminant $D$ with $1 < |D|\le \Delta$ one has $(\frac{D}{p})=-1$ for some prime factor $p$ of $n$.   Here the parameter $N(\Delta)$ may be chosen as 
$$\exp(C \log \Delta (\log \log \Delta)^2)$$ 
for some positive constant $C$.  If the Generalized Riemann Hypothesis is assumed, then $N(\Delta)$ may be chosen as $\exp(C\log \Delta \log \log  \Delta)$ for any $C> 2/\log 2$.  
\end{thm}

Since $\ell(n)$ is a square-free integer $>1$, either $\ell(n)$ or $4\ell(n)$ is a fundamental discriminant.  Therefore, if $n$ is chosen as in Theorem \ref{thm1.2} then we must have 
$$  
\ell(n) \ge \exp\Big( c \frac{\log n}{(\log \log n)^2} \Big), 
$$
for a suitable positive constant $c$.   If GRH is assumed, then the improved bound in Theorem \ref{thm1.2} permits 
$$
\ell(n) \ge \exp\Big( \Big(\frac{\log 2}{2}+ o(1)\Big) \frac{\log n}{\log \log n} \Big). 
$$
   
Theorem \ref{thm1.2} can be viewed as a variant of the classical problem in analytic number theory which for a given primitive quadratic character $\chi \bmod q$ asks for the least prime $p$ with $\chi(p)=-1$.  On GRH it is well-known (see \cite{ankenyTheLeast}) that the least such prime $p$ satisfies $p\ll (\log q)^2$.  Therefore, on GRH, the primorial $n=\prod_{p \le C(\log \Delta)^2} p$ for a suitably large constant $C$ would have the property that for each fundamental discriminant $D$ with $1<|D|\le \Delta$, some prime factor $p$ of $n$ satisfies $\chi_D(p)=-1$.  Probabilistic considerations suggest that the least prime $p$ with $\chi(p)=-1$ may lie below $C\log q (\log \log q)$.  In that case the smaller primorial $n=\prod_{p\le C \log \Delta (\log \log \Delta)} p$ would have the desired property.  Our theorem produces on GRH a square-free number of essentially this size, the key being the flexibility of not restricting to primorials; and the unconditional bound is only a little worse.  
 
\par Another motivation for this circle of problems arose from recent work of Green and Soundararajan \cite{greenCovering} which studied the threshold $\Delta$ such that the family of binary quadratic forms $x^2+ dy^2$ with $1\le d\le \Delta$ covers almost all the positive integers up to $N$.   While the problem of covering almost all integers was resolved in \cite{greenCovering}, almost nothing is known about the problem of covering {\sl all} integers below $N$.  Our work permits the following partial progress, which makes precise some results that were briefly stated in \cite{greenCovering}.   

In what follows, we will consider binary quadratic forms $f=f(x,y)=ax^2+bxy+cy^2$; such a form is called primitive if $\gcd(a,b,c)=1$.  The discriminant of such a form is $D_f = b^2-4ac$.  When $D_f<0$ the quadratic form is definite, and we shall assume then that $a>0$ so that the form is positive definite.   When $D_f>0$, the associated form is indefinite, and we shall assume that $D_f$ is not a perfect square (in which case the form is degenerate and factors into two linear forms).    

\begin{theorem} \label{thm1.3}  Let $N$ be large. 

1.   Every positive integer up to $N$ may be represented by some primitive, non-degenerate, indefinite binary quadratic form $f$ with discriminant $1< D_f \le \Delta$ with 
$$ 
\Delta = \exp\Big( (\log 4 +o(1)) \frac{\log N}{\log \log N}\Big). 
$$

2.  Assuming GRH, every positive integer up to $N$ may be represented by a primitive, positive definite binary quadratic form $f$ with discriminant $|D_f| \le \Delta$ with $\Delta$ as in part (1) above. 

3.  There exists a positive integer $n \le N$ that cannot be represented by any primitive non-degenerate binary quadratic form $f$ with $|D_f|\le \Delta$.  Here $\Delta$ may be taken as $\exp(c \log N/(\log \log N)^2)$ for some positive constant $c$.  Conditional on GRH, $\Delta$ may be taken as $\exp((\frac{\log 2}{2}-\epsilon) \log N/(\log \log N))$. 
\end{theorem}

\par Lastly, we draw attention to the work of Hanson, Vaughan, and Zhang \cite{vaughanLeast} which studies a related question of determining the least $N = N(p_1 \dots p_k)$ for which all sequences of signs of length $k$ appear as $((\frac{r}{p_1}), \dots, (\frac{r}{p_k}))$ for some $r \le N$, $\gcd(r, p_1 \dots p_k) = 1$, where $p_1, \dots, p_k$ are distinct odd primes.

\section{Proof of Theorem \ref{thm1.1}}\label{omega_Bounds}

\noindent We begin with the upper bound in Theorem \ref{thm1.1} which is a straightforward application of the pigeonhole principle.  Suppose $p_1$, $\ldots$, $p_k$ are the distinct odd primes dividing $n$.  We seek a suitably small positive integer $r$ with $r\equiv 1 \bmod 8$ and $(\frac{r}{p_j})=1$ for all $1\le j\le k$; such an integer $r$ will be a quadratic residue $\bmod \ n$. 

By the prime number theorem it follows that for a suitably large absolute constant $A$, there are at least $4 \times 2^k + k+1$ odd primes below $Ak 2^k$.  It follows that some progression $a \bmod 8$ contains at least $2^k +1$ primes below $Ak 2^k$ with these primes all being different from the $k$ odd primes dividing $n$.  By the pigeonhole principle we may find two such primes $q_1$ and $q_2$ with $(\frac{q_1}{p_j}) = (\frac{q_2}{p_j})$ for all $1\le j\le k$.   Then $r=q_1 q_2$ is $1\bmod 8$ and satisfies $(\frac{r}{p_j})=1$ for all $1\le j\le k$.  Since $r=q_1q_2$ is less  than $(Ak 2^k)^2$, the upper bound in Theorem \ref{thm1.1} follows.

The lower bound in Theorem \ref{thm1.1} is more involved, and will follow from the following proposition.  

\begin{prop} \label{prop2.1}  For each integer $k\ge 2$ there exist completely multiplicative functions $f_1$, $\ldots$, $f_k : {\mathbb N} \to \{-1,1\}$ such that the least square-free integer $r>1$ with $f_j(r)=1$ for all $1\le j\le k$ satisfies 
$$ 
r \ge 4^{k-Ck/\log k}, 
$$
for a suitable absolute constant $C$. 
\end{prop} 

Before proving the proposition, we show how the lower bound in Theorem \ref{thm1.1} may be deduced. Let $f_1$, $\ldots$, $f_k$ be completely multiplicative functions as in the proposition.   By Dirichlet's theorem on primes in progressions, for each $1\le j\le k$ we may find arbitrarily large primes $q_j$ with $(\frac {p}{q_j}) = f_j(p)$ for all primes $p <4^k$ (since quadratic reciprocity and the Chinese remainder theorem enable us to express these conditions in terms of $q_j$ lying in suitable reduced residue classes $\bmod\ {4\prod_{p< 4^k} p})$.   Now take $n=q_1 \cdots q_k$, and note that if $r =\ell(n)$ is $<4^k$ then we must have $f_j(r)=1$ for all $1\le j\le k$.  The proposition gives that $r\ge 4^{k-Ck/\log k}$, completing our proof.  

We now turn to the proof of the proposition, first recording a simple but key lemma.

\begin{lemma}\label{newtonsIdent} 
    Let $k$ be a natural number and let ${\mathbb F}_2$ denote the field with two elements, and $\overline{\mathbb F}_2$ its algebraic closure.  Suppose $z_1$, $\ldots$, $z_n$ are distinct non-zero elements of $\overline{{\mathbb F}}_2$ with 
    $$ 
    z_1^{2j-1} + z_2^{2j-1} + \ldots + z_n^{2j-1} =0
    $$
    for all $1\le j\le k$.   Then we must have $n> 2k$.  
\end{lemma}

\begin{proof}
    For $j \in \mathbb{N}$, set
    \[s_j := z_1^j + \dots + z_n^j, \quad \sigma_j := \sum_{i_1 < \dots < i_j} z_{i_1} \dots z_{i_j}\]
    and let $\sigma_0 = 1$. By assumption $s_1, s_3, \dots, s_{2k-1} = 0$, and since $(x_1 + \dots + x_n)^2 = x_1^2 + \dots + x_n^2$ in characteristic $2$, we have in fact $s_1, s_2, \dots, s_{2k} = 0$.

Recall Newton's identity, which gives for all $j\le n$
    \[
    j \sigma_j = \sum_{i=1}^j (-1)^{i-1}\sigma_{j-i}s_i. 
    \]
    If $n\le 2k$ we conclude that $j\sigma_j =0$ for all $j\le n$. Note that $\sigma_n = z_1\ldots z_n \neq 0$, and thus $n$ must be even.   Consider the polynomial 
    $$ 
    p(x) = \prod_{j=1}^{n} (x-z_j) = \sum_{j=0}^{n} \sigma_j x^{n-j} \in {\overline{\mathbb F}}_2[x],
    $$
    whose derivative satisfies (since $n$ is even)
$$
p^{\prime}(x) = \sum_{j=0}^{n-1} (n-j) \sigma_j x^{n-j-1} = 
\sum_{j=0}^{n-1} j\sigma_j x^{n-j-1} = 0.
$$
This contradicts the assumption that the roots $z_1$, $\ldots$, $z_n$ are distinct and therefore we must have $n>2k$. 
\end{proof}

\begin{proof}[Proof of Proposition \ref{prop2.1}]   We first select a bijection between the primes and irreducible polynomials in ${\mathbb F}_2[x]$ apart from the polynomial $x$ as follows.  Running over the primes in ascending order, associate to $p$ an irreducible polynomial of smallest degree in ${\mathbb F}_2[x]$ that has not already been used (and omitting $x$).  For instance, associate $2$ to the polynomial $x+1$, $3$ to $x^2+x+1$, $5$, $7$ and $11$ to the three irreducible polynomials of degree $3$ (in any way), and so on.  Let $Q_p$ denote the irreducible polynomial associated to $p$.  Since $\pi(x) \sim x/\log x$, and the number of irreducibles in ${\mathbb F}_2[x]$ of degree $d$ is $\sim 2^d/d$, we see that 
$$ 
\text{deg}(Q_p) = \frac{\log p}{\log 2} + O(1). 
$$
Let ${\mathcal Z}_p$ denote the set of roots of $Q_p$ in $\overline{\mathbb F}_2$, and note that, since the irreducible $x$ has been omitted, the different sets ${\mathcal Z}_p$ are disjoint subsets of non-zero elements of $\overline{\mathbb F}_2$.  

For each natural number $j$, and prime number $p$, define 
$$
e_j(p) = \sum_{z\in {\mathcal Z}_p} z^{2j-1},  
$$
which is an element of ${\mathbb F}_2$.   Define a completely multiplicative function $f_j$ by setting on primes $p$
$$ 
f_j(p) = (-1)^{e_j(p)}.  
$$
We claim that the proposition holds for this choice of completely multiplicative functions.  Suppose that $r\ge 2$ is a square-free number with $f_j(r)=1$ for all $1\le j\le k$.  Put ${\mathcal Z} =\cup_{p|r} {\mathcal Z}_p$, and note that by construction 
$$
\sum_{z\in \mathcal{Z}} z^{2j-1} = \sum_{p|r} \sum_{z\in \mathcal{Z}_p} z^{2j-1} = \sum_{p|r} e_j(p) = 0, 
$$
for all $1\le j\le k$.   Therefore, by Lemma \ref{newtonsIdent} we obtain 
$$ 
2k < |{\mathcal Z}| = \sum_{p|r} \text{deg} (Q_p) = \sum_{p|r} \Big(\frac{\log p}{\log 2} + O(1)\Big) = \frac{\log r}{\log 2} + O(\omega(r)),  
$$
where $\omega(r)$ denotes the number of distinct prime factors of $r$.  Since $\omega(r) =O(\log r/\log \log r)$, it follows that 
$$ 
\log r \ge k\log 4 - Ck/\log k,
$$
for a suitable absolute constant $C$, which completes our proof. 
\end{proof}

We remark that the multiplicative functions $f_j$ constructed above are related to characters on ${\mathbb F}_2[x]$ known as Hayes short interval characters.   For instance, for any non-zero polynomial $M \in {\mathbb F}_2[x]$ with $M(x) = x^n + a_1 x^{n-1} +\ldots$, the function $\chi_1(M) = (-1)^{a_1}$ is an example of a Hayes short interval character, and our function $f_1$ is related to $\chi_1$ upon making a suitable identification between the positive integers and the polynomials in ${\mathbb F}_2[x]$ that are not multiples of $x$.  For expositions on this subject, see \cite{gorodetskyCorrelation,hayesDistribution}.

\section{Proof of Theorem \ref{thm1.2}}

\noindent Throughout we shall assume that $|D|>1$ so that the trivial character is always omitted.  We begin with the result conditional on GRH.   Let $D$ denote a fundamental discriminant with $|D|\le \Delta$, and put $x=(\log \Delta)^2 (\log \log \Delta)^4$.   Let $k=\lceil \log \Delta( \frac{1}{\log 2} + \frac{A}{\log \log \Delta}) \rceil$ for a suitably large constant $A$.  Suppose that for all distinct choices of primes $p_1$, $\ldots$, $p_k$ in $[x,2x]$ there exists some fundamental discriminant $D$ with $|D|\le \Delta$ and $\chi_D(p_j)=0$ or $1$ for all $1\le j\le k$.   If $A$ is suitably large, we will obtain a contradiction to this assumption.  This implies that for some choice of distinct primes $p_1$, $\ldots$, $p_k$ in $[x,2x]$ one has for any $D$ with $|D|\le \Delta$ there exists some $j$ with $\chi_D(p_j)=-1$.   Then taking $n=p_1 \cdots p_k$ which is below $N(\Delta) = (2x)^k$,  we obtain the conditional part of the theorem.   

Our assumption shows that for distinct $p_1$, $\ldots$, $p_k$ in $[x,2x]$ we have 
$$ 
\sum_{1<|D|\le \Delta} \prod_{j=1}^{k} \Big(\frac{1+\chi_D(p_j)}{2} \Big) \ge 2^{-\log \Delta/\log x}, 
$$
since there are at most $\log \Delta/\log x$ primes larger than $x$ that may divide $D$ (these being the possibilities for $\chi_D(p_j)=0$). 
Since $\prod_{j=1}^{k} (\frac{1+\chi_D(p_j)}{2})$ is always non-negative, we conclude that 
\begin{align}
\label{3.1}
\sum_{1<|D|\le \Delta} \Big( \sum_{x<p\le 2x} \frac{1+\chi_D(p)}{2} \log p\Big)^k &= \sum_{x < p_1, \ldots, p_k \le 2x} \log p_1 \cdots \log p_k \sum_{1<|D|\le \Delta} \prod_{j=1}^k \Big(\frac{1+\chi_D(p_j)}{2}\Big) \nonumber\\
&\ge 2^{-\log \Delta/\log x} \sum_{\substack{ x< p_1, \ldots, p_k \le 2x \\ \text{distinct} }} \log p_1 \cdots \log p_k \nonumber \\ 
&\ge 2^{-\log \Delta/\log x}
\Big( x+ O\Big(\frac {x}{\log x}\Big)-k \log (2x)\Big)^k.
\end{align}

On the other hand, note that GRH gives (see, for example, Theorem 13.7 of \cite{montgomeryVaughan})
$$ 
\sum_{x< p\le 2x} \frac{1+\chi_D(p)}{2} \log p = \frac{x}{2} + O(x^{\frac 12} (\log x)\log \Delta)=x \Big(\frac 12+O\Big(\frac{1}{\log \log \Delta}\Big)\Big). 
$$
Therefore, the left side of \eqref{3.1} is 
\begin{equation*}
 |\{1<|D|\le \Delta\}| x^k \Big(\frac 12 + O\Big(\frac{1}{\log \log \Delta}\Big)\Big)^k \le 2\Delta \ x^k \Big(\frac 12 + O\Big(\frac{1}{\log \log \Delta}\Big)\Big)^k.
\end{equation*}
This contradicts \eqref{3.1} for suitably large $A$, and we conclude the conditional part of the theorem.  

The unconditional result is based on a similar idea, but the argument is more involved as we must replace the good bound for character sums over primes furnished by GRH with a substitute based on zero density estimates.   Below, it will be convenient to define the following weight function (for positive real numbers $t$) 
\begin{equation} 
\label{3.2}
w(t) = \begin{cases} 
0 &\text{if } t<2 \text{ or } t \ge 4 \\
(t-2) &\text{if } 2 \le t < 3\\ 
(4-t) &\text{if } 3\le t <4. \\
\end{cases}
\end{equation}
We begin with a lemma that distills the conditional argument above.  

\begin{lemma} \label{lem3.1}  Let ${\mathcal D}$ be a non-empty set of fundamental discriminants $D$ with $1 <|D|\le \Delta$.   Let $x\ge \log \Delta$ be such that for all $D\in {\mathcal D}$
\begin{equation}\label{3.3}
\sum_{p} \frac{\chi_D(p)}{p} \log p\  w\Big(\frac{\log p}{\log x} \Big) 
\le \frac 14 \sum_p \frac{\log p}{p} w \Big(\frac{\log p}{\log x}\Big). 
\end{equation}
Let $k= \lceil 5 \log |{\mathcal D}| + 5 \log \Delta/\log x\rceil$.   There exists an integer $n=p_1 \cdots p_k$ where the $p_j$ are distinct primes in $[x^2,x^4]$ such that for all $D\in {\mathcal D}$ there exists some $1\le j\le k$ with $\chi_D(p_j)=-1$. 
\end{lemma} 
\begin{proof} If the conclusion of the lemma does not hold, then arguing exactly as in \eqref{3.1} we obtain 
$$ 
\sum_{D\in {\mathcal D}} \Big( \sum_p \frac{1+\chi_D(p)}{2}
\frac{\log p}{p} w\Big(\frac{\log p}{\log x}\Big) \Big)^k
\ge 2^{-\log \Delta/\log (x^2)} \Big( 
\sum_p \frac{\log p}{p} w\Big(\frac{\log p}{\log x}\Big)  - k\frac{\log x^2}{x^2}\Big)^k. 
$$
A small calculation with the prime number theorem shows that 
\begin{equation}\label{3.4}
\sum_p \frac{\log p}{p} w\Big(\frac{\log p}{\log x}\Big) = \log x + O(1), 
\end{equation}
so that we may conclude 
$$ 
\sum_{D\in {\mathcal D}} \Big( \sum_p \frac{1+\chi_D(p)}{2} 
\frac{\log p}{p} w\Big(\frac{\log p}{\log x}\Big) \Big)^k
\ge 2^{-\log \Delta/\log (x^2)} ( \log x + O(1))^k.  
$$

On the other hand, by assumption the left side above is also 
$$
\le |{\mathcal D}| \Big( \frac 58 \log x + O(1)\Big)^k. 
$$
For our choice of $k$, the upper and lower bounds are contradictory, establishing the lemma.
\end{proof}

 Our next result gives a criterion in terms of a zero-free region for $L(s,\chi_{D})$ which ensures the bound \eqref{3.3}.  The proof has the flavor of Rodosskii's work on the least quadratic non-residue problem; see Theorem 1 in Chapter 9 of \cite{montgomeryTen}.

\begin{lemma} \label{lem3.2}  Let $D$ be a fundamental discriminant with $1 <|D|\le \Delta$, and let $x$ be a real number with $\Delta \ge x \ge (\log \Delta)^2$.   Suppose that the Dirichlet $L$-function $L(s,\chi_D)$ has no zeros in the region 
$$
\Big\{ s= \sigma+it: \ \ \sigma \ge 1- \frac{1}{\log x} \log \Big(\frac{10\log \Delta}{\log x}\Big), \ \ |t| \le \Delta \Big\}.  
$$
Then, for large $\Delta$, the estimate \eqref{3.3} holds: 
$$
\sum_{p} \frac{\chi_D(p)}{p} \log p\  w\Big(\frac{\log p}{\log x} \Big) 
\le \frac 14 \sum_p \frac{\log p}{p} w \Big(\frac{\log p}{\log x}\Big). $$
\end{lemma}
\begin{proof}  By Mellin inversion of the weight $w$, we find 
\begin{align*}
\frac{1}{2\pi i} \int_{1-i\infty}^{1+i\infty} -\frac{L^{\prime}}{L}(s+1,\chi_D) \Big(\frac{x^{2s}-x^s}{s}\Big)^2 ds &= (\log x) \sum_n \frac{\Lambda(n)}{n} \chi_D(n) w\Big(\frac{\log n}{\log x}\Big)\\
&=(\log x) \sum_p \frac{\log p}{p} \chi_D(p) w\Big(\frac{\log p}{\log x}\Big) + O(1), 
\end{align*}
upon estimating the contribution of prime powers trivially. Moving the line of integration to the left, we may alternatively write this as a sum over zeros of $L(s,\chi_D)$.  This yields for the above 
\begin{equation}\label{3.5}
-\sum_\rho \Big(\frac{x^{2(\rho-1)}-x^{\rho-1}}{\rho-1}\Big)^2  
+O(x^{-2}), 
\end{equation}
where the error term accounts for the contribution of trivial zeros (which are a subset of the non-positive integers). 

Since (for any $T\ge 1$) there are $\ll T \log (T|D|)$ non-trivial zeros $\rho$ of $L(s,\chi_D)$ with $|\text{Im }\rho| \le T$, the contribution of the terms with $|\text{Im}\rho| \ge \Delta$ to \eqref{3.5} is easily bounded by $\ll (\log \Delta)/\Delta$.   For the zeros with $|\text{Im} \rho| \le \Delta$, the assumed zero-free region shows that the contribution of these terms to \eqref{3.5} is bounded in size by 
$$ 
\le 2 \Big(\frac{\log x}{10\log \Delta}\Big)^2 \sum_{|\text{Im} \rho| \le \Delta} \frac{1}{|\rho-1|^2} \le 2 \Big(\frac{\log x}{10\log \Delta}\Big)^2 \sum_{|\text{Im}\rho| \le \Delta} (\log x) \text{Re }\frac{1}{1+1/\log x -\rho},
$$
where the last step follows because (writing $\rho = \beta+i\gamma$ and noting that $\beta \le 1- 2/\log x$ by assumption)
$$
\text{Re} \frac{1}{1+1/\log x - \rho} = \frac{(1+1/\log x -\beta)}{(1+1/\log x -\beta)^2 + \gamma^2} \ge \frac{3/\log x}{(\frac 32 (1-\beta))^2 +\gamma^2} \ge \frac{1}{|\rho-1|^2 \log x }. 
$$
We conclude that the quantity in \eqref{3.5} is bounded in size by  
\begin{equation}\label{3.6}
\le 2 \Big(\frac{\log x}{10\log \Delta}\Big)^2 (\log x) \sum_{\rho} \text{Re} \frac{1}{1+1/\log x -\rho} +O\Big(\frac{\log \Delta}{\Delta}+x^{-2}\Big).
\end{equation}

Taking logarithmic derivatives in the Hadamard factorization formula (see Corollary 10.18 of \cite{montgomeryVaughan}) gives 
$$
 \sum_{\rho}\text{Re} \frac{1}{1+1/\log x -\rho} = \tfrac 12 \log |D| +O(1) + \frac{L^{\prime}}{L} (1+1/\log x, \chi_D) \le 2 \log \Delta + O(1),
 $$
upon bounding $L^{\prime}/L(1+1/\log x, \chi_D)$ trivially by $-\zeta^{\prime}/\zeta(1+1/\log x) = \log x + O(1)$. 
Using this in \eqref{3.6} we conclude that 
$$
\sum_{p} \frac{\log p}{p} \chi_D(p) w\Big(\frac{\log p}{\log x}\Big)
\le 2\Big(\frac{\log x}{10\log \Delta}\Big)^2 (2\log \Delta +O(1)) + 
O\Big(\frac{1}{\log x} + \frac{\log \Delta}{\Delta}\Big). 
$$
Since $(\log \Delta)^2 \le x \le \Delta$, for large enough $\Delta$ the desired estimate now follows from \eqref{3.4}.
\end{proof}

We are now ready to prove the unconditional part of the theorem.  Put $J= \lceil \log(\frac 12 \log \Delta)\rceil$ and for $1\le j\le J$ let $\beta_j =1 - e^{j}/\log \Delta$.  For $1\le j\le J-2$, define ${\mathcal D}_j$ to be the set of fundamental discriminants $D$ with $1<|D|\le \Delta$ such that $L(s,\chi_D)$ has no zeros in the region $\{\sigma +it: \ \sigma \ge \beta_j, |t|\le \Delta\}$ but does have a zero in the region $\{ \sigma+it: \ \sigma \ge \beta_{j+1}, |t| \le \Delta\}$.  Let ${\mathcal D}_0$ denote the set of fundamental discriminants $D$ with $1<|D|\le \Delta$ which have a zero in the region $\{\sigma \ge \beta_1, \ |t|\le \Delta\}$.  Finally define ${\mathcal D}_{J-1}$ to be any remaining fundamental discriminants with $1<|D|\le \Delta$ that are not included in the sets ${\mathcal D}_0$, $\ldots$, ${\mathcal D}_{J-2}$.   The sets ${\mathcal D}_0$, $\ldots$, ${\mathcal D}_{J-1}$ clearly partition the set of all fundamental discriminants $D$ with $1<|D|\le \Delta$.  By a standard log-free zero density estimate (see, e.g., \cite{gallagherLarge}), we have 
\begin{equation}\label{zeroDensityDj}
|\mathcal{D}_j| \le C_1 \exp( C_2 e^j),  
\end{equation}
for some absolute constants $C_1$ and $C_2$ and all  $0 \le j \le J-1$ (we may assume that $C_2\ge 4e$ so that the bound for ${\mathcal D}_{J-1}$ is trivial, being much larger than $\Delta$).

For each $1\le j\le J-1$ put $x_j = \Delta^{j/e^{j-1}}$.  Note that $(\log \Delta)^2 \le x_j \le \Delta$, and that the fundamental discriminants in ${\mathcal D}_j$ satisfy the zero-free region required in Lemma \ref{lem3.2} and therefore the conclusion there holds (with $x=x_j$ there).   If ${\mathcal D}_j$ is empty, take $n_j=1$.  If ${\mathcal D}_j$ is non-empty, Lemma \ref{lem3.1}  furnishes an integer $n_j$ with at most $5(\log C_1 + C_2 e^j + e^{j-1}/j)+1$ prime factors all in $[x_j^2,x_j^4]$ such that for each $D\in {\mathcal D}_j$ there is some prime factor $p$ of $n_j$ with $\chi_D(p)=-1$.  Note that $n_j$ satisfies the bound 
\begin{equation}
    \label{3.8}
   \log  n_j \le 20 \Big( \frac{j}{e^{j-1}}\log C_1 + e\cdot  j(C_2+1)\Big) \log \Delta.
\end{equation}

There remain a bounded number of fundamental discriminants $D$ in ${\mathcal D}_0$.  For each such $D$, we may easily find an odd prime $p\le (2\Delta)$ with $\chi_D(p)=-1$.  (To see this, note that if $\chi_D(2)=-1$ then $|D|+2$ is odd and satisfies $\chi_D(|D|+2)=-1$ so that some odd prime factor of $|D|+2$ works.  If $\chi_D(2)\neq -1$ then there must be some odd integer $a$ with $1\le a\le |D|$ with $\chi_D(a)=-1$ and some prime factor of $a$ works.)  Take $n_0$ to be the product of such primes $p$, so that $n_0 \le (2\Delta)^{C_1 e^{C_2}}$.  

Define $n$ to be the product of the primes dividing $n_0 n_1\cdots n_{J-1}$.  Then $n$ is odd and square-free and by \eqref{3.8} satisfies the bound 
$$ 
\log n \le (\log 2\Delta) \Big( C_1 e^{C_2} + 80 \log C_1 + 40 J^2 (C_2+1)\Big).
$$
Further, by construction, for each fundamental discriminant $D$ with $1<|D|\le \Delta$ there exists some prime factor $p$ of $n$ with $\chi_D(p)=-1$.  This completes our proof.

\section{Representations by binary quadratic forms: Proof of Theorem \ref{thm1.3}}\label{applications}

\noindent We begin by recalling some classical facts on representing integers by binary quadratic forms.   If $f$ is a primitive binary quadratic form of discriminant $Dt^2$ where $D$ is a fundamental discriminant, then there is an integral linear transformation $T$ with determinant $t$ and a form ${\tilde f}$ of discriminant $D$ such that $f(x,y)= {\tilde f}(T(x,y))$ (see Proposition 7.1 of \cite{buellBinary}).   Note that the form ${\tilde f}$ is automatically primitive, since its discriminant is fundamental.  Thus the integers represented by primitive binary quadratic forms of discriminant $Dt^2$ are also represented by (primitive) binary quadratic forms of discriminant $D$. This observation allows us to restrict attention to binary quadratic forms with fundamental discriminants.

Secondly, if $D$ is a fundamental discriminant, then the positive integer $n$ is properly represented by some binary quadratic form of discriminant $D$ if and only if $D$ is a square $\bmod \ {4n}$ (see Proposition 4.1 of \cite{buellBinary} but beware that the restriction to proper representations is inadvertently omitted there).  Here, a representation of $n$ as $f(x,y)$ is called proper if $\gcd(x,y)=1$.  Note that if $n$ is square-free then any representation is automatically proper.

Consider the first statement in Theorem \ref{thm1.3}.  For any positive integer $n$, note that $\ell(4n)$ is square-free and must be $1\bmod 4$ and is therefore a fundamental discriminant.  If we denote this fundamental discriminant by $D$, then $D$ is a square  $\bmod \ {4n}$, and therefore $n$ may be represented by some indefinite binary quadratic form of discriminant $D$.   The first statement of Theorem \ref{thm1.3} follows from the bound \eqref{1.2} (which followed from the upper bound in Theorem \ref{thm1.1}).   

 Next consider the third statement in Theorem \ref{thm1.3}.  Theorem \ref{thm1.2} produces a square-free integer $n$ with the property that for all fundamental discriminants $D$ with $1<|D|\le \Delta$ one has $\chi_D(p)=-1$ for some prime factor $p$ of $n$.  This means that $D$ is not a square $\bmod \ n$, and therefore $n$ cannot be represented by any binary quadratic form of discriminant $D$.  The third statement in Theorem \ref{thm1.3} follows from the bounds on $n$ given in Theorem \ref{thm1.2}.  

 Lastly consider the second statement in Theorem \ref{thm1.3}.  Let $n \le N$ be given, and we wish to find (assuming GRH) a negative fundamental discriminant $D$ with $0 < -D =|D|\le \Delta$ such that $n$ is represented by a binary quadratic form of discriminant $D$.  If a number $m$ is represented by such a binary quadratic form, then by scaling clearly $mt^2$ would also be represented by the same binary quadratic form.  We may therefore assume that $n$ is square-free, and write $n=2^a m$ with $a=0$ or $1$ and $m$ odd and square-free.  Further, we will restrict attention to fundamental discriminants $D$ of the form $D=-\ell$ with $\ell$ being a prime $\equiv 7 \bmod 8$.  Since such $D=-\ell$ are $1 \bmod {8}$, it is clear that $D$ is a square modulo $4 \times 2^a$, and we need only find such $D$ that are squares modulo the odd part $m$.  Now note that $\sum_{d|m} (\frac{-\ell}{d}) = 0$ unless $-\ell$ is a square modulo ${m}$.  With this in mind, consider 
 $$
 \sum_{\substack{ \ell \le \Delta \\ \ell\equiv 7 \bmod 8}} \log \ell \Big(\sum_{d|m} \Big(\frac{-\ell}{d}\Big) \Big)
 = \sum_{d|m} \sum_{\substack{ \ell \le \Delta \\ \ell\equiv 7 \bmod 8}} \log \ell \Big(\frac{-\ell}{d}\Big). 
 $$
The term $d=1$ gives $\tfrac 14 \Delta +O(\sqrt{\Delta} (\log \Delta)^2)$ (since GRH and therefore RH is assumed), and the terms with $d>1$ (and there are $2^{\omega(m)}-1$ such terms) each give $O(\sqrt{\Delta} (\log n\Delta)^2)$ (by GRH).   Thus the above is 
$$ 
\frac 14 \Delta + O( 2^{\omega(n)} \sqrt{\Delta} (\log n\Delta)^2), 
$$
so that if $\Delta$ is larger than $C 4^{\omega(n)} (\log n)^4$ for a suitable constant $C$, then there exists a prime $\ell \le \Delta$ with $\ell \equiv 7 \bmod 8$ and $-\ell$ being a square  $\bmod\  {4n}$. Since $\omega(n) \le (1+o(1)) \log N/\log \log N$ for all $n\le N$, the second assertion in Theorem \ref{thm1.3} follows.

\section{Further questions}

Let us flesh out a question implicit in our work in Section 2.   Given a natural number $k$, determine the smallest positive integer $L(k)$ such that for all choices of $k$ completely multiplicative functions $f_j: {\mathbb N} \to \{-1, 1\}$ (with $1\le j\le k$) there exists a square-free integer $2\le r\le L(k)$ with $f_1(r)=f_2(r)=\ldots = f_k(r)=1$.  

If ${\mathcal A}$ is any set of $2^{k}+1$ square-free integers, then there must exist two distinct elements $a$, $b$ in ${\mathcal A}$ with $f_j(a)=f_j(b)$ for all $1\le j\le k$.   Then $r= ab/(a,b)^2$ is a square-free integer larger than $1$ with $f_j(r)=1$ for all $1\le j\le k$.   Thus $L(k)$ is bounded above by $\max_{a\neq b \in {\mathcal A}} ab/(a,b)^2$ for any such set ${\mathcal A}$, and we can try and minimize this quantity over possible sets ${\mathcal A}$.   One choice for ${\mathcal A}$ is simply to take the first $2^{k}+1$ square-free integers (starting with $1$).  This example gives $L(1)\le 6$, $L(2)\le 30$ and $L(3)\le 143$, and in these cases equality is attained by means of the following examples.  For $k=1$ consider a  completely multiplicative function $f_1$ with $f_1(2)=f_1(3)=f_1(5)=-1$.  For $k=2$ consider  
\begin{center}
\begin{tabular}{ |c|c|c|c|c| } 
\hline
  & $2$ & $3$ & $5$ & $p \ge 7$ \\ 
   \hline
 $f_1$ & $1$ & $-1$ & $-1$ & $-1$ \\ 
 $f_2$ & $-1$ & $1$ & $-1$ & $-1$ \\ 
 \hline
\end{tabular}
\end{center}
For $k=3$ consider 
\begin{center}
\begin{tabular}{ |c|c|c|c|c|c|c|c|c| } 
\hline
  & $2$ & $3$ & $5$ & $7$ & $p \ge 11$ \\ 
   \hline
 $f_1$ & $1$ & $1$ & $-1$ & $-1$ & $-1$\\ 
 $f_2$ & $1$ & $-1$ & $1$ & $-1$ & $-1$\\
 $f_3$ & $-1$ & $1$ & $1$ & $-1$ & $1$ \\
 \hline
\end{tabular}
\end{center}
When $k=4$, a better choice for ${\mathcal A}$ is $\{ 1, 2, 3, 5, 6, 7, 10, 11, 13, 14, 15, 17, 19, 21, 22, 26, 30\}$ which leads to the upper bound $L(4) \le 570$ (which is a little better than simply taking the first $17$ square-free numbers).    We found that in fact  $L(4)=570$, and also determined that $L(5)=2679$.   

Our work in Proposition \ref{prop2.1} shows that $L(k) \ge 4^{k-Ck/\log k}$.    Since the number of square-free integers up to $x$ is $\sim \frac{6}{\pi^2} x$, the upper bound furnished by taking ${\mathcal A}$ to be the first $2^k+1$ square-free integers is asymptotically  $L(k) \le (\frac{\pi^4}{36}+o(1)) 4^k$.

\par Another question of independent interest concerns the minimum proportion of integers $n\le x$ (with $x$ large) for which $f_1(n) = \dots = f_k(n) = 1$ where $f_1, \dots, f_k: \mathbb{N} \to \{ -1, 1 \}$ are completely multiplicative functions. When $k = 1$, Granville and Soundararajan \cite{granvilleSpectrum} showed that for large $x$ 
\[
| \{ n\le x: f_1(n)=1\} | \ge (\delta_1+o(1))x,
\]
uniformly over all completely multiplicative functions $f_1:\mathbb{N}\to \{-1, 1\}$, 
where 
\[
\delta_1 = 1 - \log(1 + \sqrt{e}) + 2 \int_1^{\sqrt{e}} \frac{\log t}{t+1}\,dt = 0.171500\dots .
\]
Equality is attained essentially when 
\[f_1(p) = \begin{cases}
    1 & \text{ for }\hspace{18pt}p \le x^{1/(1+\sqrt{e})} \\
    -1 & \text{ for }\hspace{5pt}x^{1/(1+\sqrt{e})} < p \le x.
\end{cases}\]

\par Similarly, it can be shown that there exists a constant $\delta_2 > 0$ for which
\[
|\{n \le x: f_1(n)=f_2(n)=1\}|  \ge (\delta_2+o(1))x, 
\]
uniformly for all completely multiplicative functions $f_1, f_2: \mathbb{N} \to \{ -1, 1\}$. The best upper bound for $\delta_2$ that we found through numerical investigations is $0.068921\dots$ which is attained by the example 
\[
f_1(p) = \begin{cases}
    1 & \text{ for }\hspace{18pt}p \le x^{1/\alpha} \\
    -1 & \text{ for }\hspace{5pt}x^{1/\alpha} < p \le x
\end{cases} \hspace{15pt} f_2(p) = \begin{cases}
    1 & \text{ for }\hspace{23pt}p \le x^{1/\alpha\beta} \\
    -1 & \text{ for }\hspace{5pt}x^{1/\alpha\beta} < p \le x^{1/\alpha} \\ 
    \pm 1 &\text{ for } \hspace{10pt}x^{1/\alpha} < p \le x
\end{cases}\]
with $\alpha = 2.521469\dots$ and $\beta = 2.796705\dots$; in the range $p>x^{1/\alpha}$ the values of $f_2(p)$ alternate in sign as $p$ increases (or are chosen randomly to be $\pm 1$ with equal probability).  It may be interesting to work out a decent lower bound for $\delta_2$, or to see if better upper bounds exist than the above example. 

\par For all $k$ one can show the existence of a constant $\delta_k >0$ such that  
\[
|\{ n\le x: f_1(n)= f_2(n) =\ldots = f_k(n)=1 \} | \ge (\delta_k + o(1))x, \] 
uniformly for all completely multiplicative functions $f_1, \dots, f_k: \mathbb{N} \to \{ -1, 1\}$, and it may be of interest to determine good lower bounds for $\delta_k$.

\bibliographystyle{unsrt}
\bibliography{main}{}

\end{document}